\documentclass[12pt]{amsart}

\usepackage{geometry}
\usepackage{amsthm}
\usepackage{amssymb}
\usepackage{microtype}
\usepackage[dvipsnames]{xcolor}
\usepackage[unicode,colorlinks=true,linktocpage=true,citecolor=Emerald,linkcolor=Fuchsia]{hyperref}
\usepackage[capitalise]{cleveref} 
\usepackage{not_clever}
\usepackage{tikz-cd}
\usepackage{mathrsfs}
\theoremstyle{plain}
\newtheorem{theorem}{Theorem}[section]
\newtheorem{proposition}[theorem]{Proposition}
\newtheorem{lemma}[theorem]{Lemma}

\theoremstyle{definition}
\newtheorem{example}[theorem]{Example}

\newcommand{\F}{\mathcal{F}}
\newcommand{\res}{\operatorname{res}}
\newcommand{\tr}{\operatorname{tr}}
\newcommand{\im}{\operatorname{im}}
\newcommand{\Aut}{\operatorname{Aut}}

\newcommand{\gen}[1]{\langle #1 \rangle}
\renewcommand{\phi}{\varphi}
\newcommand{\xra}{\xrightarrow}

\numberwithin{equation}{section}
\makeatletter
\let\c@theorem\c@equation
\makeatother

\begin{document}

\title{Splitting the center of a Sylow subgroup} 

\author{George Glauberman}
\email{gglauber@uchicago.edu}

\author{Justin Lynd}
\email{lynd@louisiana.edu}

\thanks{The authors would like to thank the Isaac Newton Institute for
Mathematical Sciences, Cambridge, for support and hospitality in July 2022
during the programme Groups, representations and applications where work on
this paper was undertaken. This work was supported by EPSRC grant no
EP/K032208/1. In addition, GG was partially supported by Simons Collaboration
grant 426819, and JL was partially supported by NSF Grant DMS-1921550 and
Simons Foundation Grant SFI-MPS-TSM-0001418. We thank these institutions for
their support.}

\begin{abstract}
Suppose $p$ is a prime and $S$ is a Sylow $p$-subgroup of a finite group $G$. 
If $S$ is normal in $G$, then $Z(S)$ is the direct product of $S \cap Z(G)$ with $[Z(S), G]$. 
We prove an analogous result for all groups except in some cases where $p=2$ and $G$ is not solvable, where we have counterexamples. 
We also extend this result to fusion systems.
\end{abstract}

\maketitle

\section{Introduction}
A finite group $G$ is said to split over a normal subgroup $N$ if there exists a complement $H$ to $N$ in $G$, i.e., a subgroup $H$ such that $HN = G$ and $H \cap N = 1$. 
In this note, we show if $S$ is a Sylow $p$-subgroup of $G$, then $Z(S)$ splits over a subgroup closely related to $Z(S) \cap Z(G)$, except for some cases when $p = 2$ and $G$ is not solvable.
The proof of this is not so difficult modulo known results, but the generality in which it holds initially surprised us. 
Counterexamples at the prime $2$ also seem to be fairly restricted and suggest that it might hold in wider generality than what we prove here.

In order to motivate the result, consider the special case where $S$ is normal in $G$.
Then $Z(S)$ is characteristic in $S$ and hence normal in $G$.
Since $S$ centralizes $Z(S)$, conjugation in $G$ induces an action of the $p'$-group $G/S$ on $Z(S)$. 
Hence, by results on coprime action \cite[Theorem~5.2.3]{Gorenstein1980},
\begin{equation}
Z(S) = (Z(S) \cap Z(G)) \times [Z(S), G]. 
\label{S normal splitting}
\end{equation}
Of course it is possible that one of the factors is just the identity subgroup.

Assume from now on the general situation that 
\begin{equation}
\text{$G$ is a finite group, $p$ is a prime, and $S$ is a Sylow $p$-subgroup of $G$.}
\label{gps}
\end{equation}
As usual, let $x^g$ and $^gx$ denote the conjugates $g^{-1}xg$ and $gxg^{-1}$ for $x \in S$ and $g \in G$. 
Our main result is a generalization of \eqref{S normal splitting}. 
An element $x$ of $S$ is \emph{weakly closed} in $S$ with respect to $G$ if the only conjugate of $x$ in $G$ that lies in $S$ is $x$ itself, i.e., 
\[
\text{$x^g = x$ whenever $g \in G$ and $x^g \in S$.}
\]

Let $x, y$ be two weakly closed elements, $z = xy$, and $g \in G$ with $z^g \in S$. 
Then $\gen{^gx,{}^gy} \leq C_G(z)$, so by Sylow's theorem there is $h \in C_G(z)$ with $^h\gen{^gx,{}^gy} \leq S$. 
Thus ${ }^{hg}z = xy = z$ by weak closure of $x$ and $y$, and so $z = z^g$. 
This shows that the collection $W_G(S)$ of all elements that are weakly closed in $S$ with respect to $G$ is a subgroup of $Z(S)$. 

Recall that $O_p(G)$ denotes the largest normal $p$-subgroup of $G$, 
$O_{p'}(G)$ denotes the largest normal $p'$-subgroup of $G$, and $O_{p',p}(G)$ and $Z^*_p(G)$ are the normal subgroups of $G$ that contain $O_{p'}(G)$ and satisfy $O_{p',p}(G)/O_{p'}(G) = O_p(G/O_{p'}(G))$ and $Z_p^*(G)/O_{p'}(G) = Z(G/O_{p'}(G))$, respectively.  

Our main result asserts that in many cases, $W_G(S)$ is a direct factor of $Z(S)$. 
For example, it is easy to see that if $S$ is normal in $G$, then $W_G(S)$ is just the direct factor $Z(S) \cap Z(G)$ in \eqref{S normal splitting}. 

\begin{theorem}
Assume \eqref{gps} and one of the following:
\begin{enumerate}
\item $p$ is odd, or
\item $Z(S) \leq O_{2',2}(G)$, or
\item $G$ is solvable. 
\end{enumerate}
Then $W_G(S)$ is a direct factor of $Z(S)$. 
\label{thm WGS}
\end{theorem}

The proof of Theorem~\ref{thm WGS} is divided into two parts.
In each case, we first show that $Z(S)$ splits over $W_H(S)$ for a well-chosen subgroup $H$ of $G$ containing $S$ and having $Z(S) \leq O_{p',p}(H)$, by reducing to the case $Z(S) \leq O_p(H)$ and using the transfer map $\tr_S^H$ on fixed points.
Then we show that $W_G(S) = W_H(S)$. 
In each case, 
\[
Z(S) = \ker(\tr_S^H) \times W_G(S) \text{ and } \im(\tr_S^H) = W_G(S). 
\]
Case (3) is a special case of case (2). In case (1) we use a result \cite[Theorem~14.4]{Glauberman1971} of the first author to identify weakly closed elements. 
Our counterexamples for $p = 2$ and non-solvable groups are inspired by our previous work \cite{GlaubermanLynd2016, GlaubermanLynd2021} on the relationship between automorphisms of a group and of its fusion system. 

We note that 
\[
W_G(S) = Z(S) \cap Z_p^*(G) 
\] 
by the $Z_p^*$-theorem \cite{Glauberman1966}, \cite[Remark~7.8.3]{GLS3}. 
This requires the Classification of Finite Simple Groups when $p$ is odd, but none of our results depend on this classification.

The reader familiar with fusion systems will realize that Theorem~\ref{thm WGS} is really a statement about the fusion system $\F_S(G)$ of the group $G$ on the Sylow subgroup $S$ \cite{AschbacherKessarOliver2011}.
The argument used to prove Theorem~\ref{thm WGS} works just as well in the more general setting of saturated fusion systems.
\begin{theorem}
Let $p$ be a prime and $\F$ a saturated fusion system on the finite $p$-group $S$.
If
\begin{enumerate}
\item $p$ is odd, or
\item $Z(S) \leq O_p(\F)$,
\end{enumerate}
then $Z(\F)$ is a direct factor of $Z(S)$. 
\label{thm ZF}
\end{theorem}
Here, $Z(\F)$ is the center of the fusion system, the subgroup of $Z(S)$ consisting of those elements $z$ such that any morphism in the fusion system has an extension defined on and fixing $z$. 
A short proof shows that $Z(\F) = W_G(S)$ when $\F$ is the fusion system $\F_S(G)$ of a finite group $G$ \cite[Lemma~I.4.2]{AschbacherKessarOliver2011}. 

We prove Theorems~\ref{thm WGS} and \ref{thm ZF} in Section~\ref{proofs}. 

We are especially pleased that this result about splitting at the ``bottom'' of a $p$-group $S$ is related to a joint transfer theorem of I. M. Isaacs with S. M. Gagola \cite{GagolaIsaacs2008} about splitting at the ``top''. 
This article is dedicated to the memory of I. M. Isaacs, in appreciation of many years of friendship, help, and encouragement as a classmate, colleague, and friend of the first author.

\section{Proofs and example}
\label{proofs}
The first lemma can be viewed as a consequence of the fact that the quotient map $G \to G/O_{p'}(G)$ induces an isomorphism $\F_S(G) \cong \F_S(G/O_{p'}(G))$ of fusion systems \cite[Theorem~2.12]{Linckelmann2007}.
We give some background on fusion systems a little farther down.

\begin{lemma}
Let $N$ be a normal $p'$-subgroup of the finite group $H$. 
The quotient map $H \to H/N$ induces isomorphisms $Z(S) \cong Z(SN/N)$ and $W_H(S) \cong W_{H/N}(SN/N)$. 
\label{weakly closed mod Op'}
\end{lemma}
\begin{proof}
The statement on centers follows from the isomorphism $SN/N \cong S$ as $N$ is of order prime to $p$. 
If $zN \in W_{H/N}(S)$ and $h \in H$ is such that $z^h \in S$, then $(zN)^{hN} \in SN/N$ and so $(zN)^{hN} = zN$.
Thus, $[z,h] = z^{-1}z^h \in S \cap N = 1$, so $z^h = z$ and $z \in W_H(S)$.

Conversely, if $z \in W_H(S)$ has $(zN)^{hN} \in SN/N$, then $z^h \in SN$. 
So by Sylow's Theorem there is some $x \in N$ such that $z^{hx} \in S$ and hence $z^{hx} = z$.
Therefore, $(zN)^{hN} = zN$, and this shows that $zN \in W_{H/N}(SN/N)$. 
\end{proof}

The next proposition handles case (2) of Theorem~\ref{thm WGS}. 
If $H$ is a finite group with subgroup $K$ and $M$ is an $H$-module with left action $m \mapsto { }^hm$, we use $\tr_K^H \colon C_M(K) \to C_M(H)$ for the usual transfer map on fixed points, given by
\[
\tr_K^H(m) = \sum_{t \in [H/K]} { }^tm,
\]
the sum over a set of representatives for the left cosets of $K$ in $H$. 

\begin{proposition}
Let $H$ be a finite group, $p$ a prime, and $S$ a Sylow $p$-subgroup of $H$. 
If $Z(S) \leq O_{p',p}(H)$, then $Z(S)$ splits over $W_H(S)$. 
\label{constrained case}
\end{proposition}
\begin{proof}
By Lemma~\ref{weakly closed mod Op'}, we may as well assume that $O_{p'}(H) = 1$. 
Let $Q = O_p(H)$ and $M = Z(Q)$. 
Then $M$ is a $\mathbb{Z}_{(p)}[H]$-module via conjugation. 
The assumptions that $O_{p'}(H) = 1$ and $Z(S) \leq O_p(H)$ give $C_M(S) = Z(S)$ and $C_M(H) = W_H(S) = Z(H)$. 
The composite $\tr_S^H\res_S^H \colon C_M(H) \to C_M(S) \to C_M(H)$ raises elements to the power $|H:S|$, prime to $p$.
So the composite is an isomorphism, and it then follows for formal reasons (e.g., \cref{formal ker tr} below) that $C_M(S) = \im(\res_S^H) \times \ker(\tr_S^H) = W_H(S) \times \ker(\tr_S^H)$. 
\end{proof}

The proof of \cref{constrained case} needed the next lemma, whose proof is easy and omitted. 
\begin{lemma}
\label{formal ker tr}
Let $M \xra{r} N \xra{t} M$ be homomorphisms of modules for a ring. If $t \circ r$ is an isomorphism,
then $N = \im(r) \oplus \ker(t)$. 
\end{lemma}

When $S$ is a finite $p$-group, define $d(S)$ to be the maximum order of an abelian subgroup of $S$, and $\mathcal{A}(S)$ to be the set of abelian subgroups of $S$ of order $d(S)$. 
In this article, the Thompson subgroup $J(S)$ of $S$ is the subgroup generated by $\mathcal{A}(S)$. 
(This is one of several related subgroups called the ``Thompson subgroup'' in the literature.)

\begin{theorem}[{\cite[Theorem~14.4]{Glauberman1971}}]
\label{control of wc}
Fix a finite group $G$, a Sylow $p$-subgroup $S$ of $G$, and set $H = N_G(J(S))$. 
If $p$ is odd, then $W_G(S) = W_H(S) = O_p(Z(H))$. 
\end{theorem}

\begin{proof}[Proof of Theorem~\ref{thm WGS}]
If $G$ is solvable (i.e., (3)), then case (2) holds by \cite[Theorem~6.3.3]{Gorenstein1980}, and 
Proposition~\ref{constrained case} applies if we are in the situation of (2). 
When $p$ is odd, $W_G(S) = W_H(S)$ by Theorem~\ref{control of wc}, and so as $Z(S) \leq ZJ(S) \leq J(S) \leq O_{p',p}(N_G(J(S)))$, we see that (2) implies that $Z(S)$ splits over $W_H(S) = W_G(S)$, which gives (1). 
\end{proof}

For a Sylow $p$-subgroup $S$ of a group $G$, the fusion system $\F_S(G)$ of $G$ on $S$ is the category whose objects
are the subgroups of $S$ and whose morphisms are the functions $\phi \colon P \to Q$ between two such subgroups for which there exists an element $g$ of $G$ such that $\phi(x) = { }^gx$ for every $x \in P$.
The special features of this category were axiomatized by Puig, and later by Broto--Levi--Oliver and Roberts--Shpectorov, into what is now called a saturated fusion system $\F$ on a finite $p$-group $S$ \cite{Puig2006, BrotoLeviOliver2003, RobertsShpectorov2009}.

The conclusions of both Theorem~\ref{thm WGS} and Theorem~\ref{thm ZF} are false without any restriction on $G$ or $\F$. 
One way to produce counterexamples of groups and fusion systems is to start off with a finite simple group having a non-inner automorphism of order 2 that centralizes a Sylow $2$-subgroup. 
Or if working in the fusion system setting, start with a fusion system at the prime $2$ whose centric linking system has a non-inner rigid automorphism of order 2 \cite{GlaubermanLynd2021}.
In fact, we do not know how to construct any counterexamples other than by this method.
We illustrate such a construction with $A_6$ below; see \cite{VillarealRigid} for additional examples of groups and fusion systems $\F$ to which this construction would apply. 
\begin{example}\label{example}
Take $p = 2$ and $H = A_6$. 
Let $\gen{a}$ be cyclic of order $4$ acting on $H$ via conjugation by the transposition $(5,6)$. 
Consider the semidirect product $G = H\gen{a}$ and the Sylow $2$-subgroup $S = \gen{(1,3)(2,4),(1,2)(5,6), a}$ of $G$, set $z = (1,2)(3,4)$, 
and let $\F = \F_S(G)$. 
Then $Z(S) = \gen{z,a} \cong C_2 \times C_4$ does not split over $Z(\F) = W_G(S) = \gen{a^2}$.
\end{example}

Fix a saturated fusion system $\F$ on $S$ and let $P$ be a subgroup of $S$.
An $\F$-\emph{conjugate} of $P$ is just a subgroup isomorphic to $P$ in the category $\F$. 
A subgroup $P$ of $S$ is \emph{fully $\F$-normalized} if $|N_S(P)| \geq |N_S(P')|$ for each $\F$-conjugate $P'$ of $P$ and $\F$-\emph{centric} if $C_S(P') = Z(P')$ for each $\F$-conjugate $P'$ of $P$. 
It is fully \emph{$\F$-automized} if $\Aut_S(P)$ is a Sylow $p$-subgroup of $\Aut_\F(P)$.
One of the axioms for saturation includes the condition that every fully normalized subgroup is fully automized \cite[Theorem~I.2.5(i)]{AschbacherKessarOliver2011}. 
For example, every normal subgroup of $S$ is both fully normalized and fully automized.

The \emph{normalizer} $N_\F(P)$ is the fusion subsystem on $N_S(P)$ consisting of those morphisms $\phi \colon Q \to R$ in $\F$ that extend to a morphism $\tilde{\phi}\colon PQ \to S$ leaving $P$ invariant. 
A normal subgroup $P$ of $S$ is \emph{normal in} $\F$ if $N_\F(P) = \F$.
The unique largest normal subgroup of $\F$ is denoted $O_p(\F)$. 

\begin{proof}[Proof of Theorem~\ref{thm ZF}]
Suppose first that $Z(S) \leq Q := O_p(\F)$, so that $Z(S) \leq Z(Q)$. 
Then $Z(Q)$ is normal in $\F$, so $Z(Q)$ is fully $\F$-normalized. 
Since $\F$ is saturated, $T := \Aut_S(Q)$ is a Sylow $p$-subgroup of $H = \Aut_\F(Q)$, 
and $Z(Q)$ is an $\mathbb{Z}_{(p)}[H]$-module with
\[
Z(S) = C_{Z(Q)}(T) \text{\quad and \quad} Z(\F) = Z(N_\F(Z(Q))) = C_{Z(Q)}(H).
\]
The composite $\tr_T^H \res_T^H \colon Z(\F) \to Z(S) \to Z(\F)$ raises elements of $Z(\F)$ to the power $|H:T|$.
So since $T$ has index prime to $p$ in $H$, again $Z(S) = Z(\F) \times \ker(\tr_T^H)$ by \cref{formal ker tr}. 

Next assume $p$ is odd. 
The Thompson subgroup $J(S)$ is $\F$-centric and normal in $S$. 
By a result of Puig \cite[Theorem~I.5.5]{AschbacherKessarOliver2011}, the normalizer fusion system $N_\F(J(S))$ is a saturated fusion system on $N_S(J(S)) = S$ with
$Z(S) \leq C_S(J(S)) = Z(J(S)) \leq J(S) \leq O_p(N_\F(J(S)))$. 
So by the previous paragraph, we have
\[
Z(S) = Z(N_\F(J(S))) \times \ker(\tr_T^H)
\]
where $T = \Aut_S(J(S))$ and $H = \Aut_\F(J(S))$. 
Finally, \cite[Theorem~4.1]{DiazGlesserMazzaPark2009}, which is a fusion system analogue of Theorem~\ref{control of wc}, shows that $Z(\F) = Z(N_\F(J(S)))$ and hence completes the proof of the theorem. 
\end{proof}

\bibliographystyle{amsalpha}{ }
\bibliography{/home/justin/work/math/research/mybib.bib}
\end{document}